\documentclass[12pt]{amsart}
\usepackage{graphicx}
\advance\hoffset by -0.5 truecm
\usepackage{amssymb}
\newtheorem{Theorem}{Theorem}[section]
\newtheorem{Lemma}[Theorem]{Lemma}

\newtheorem{Proposition}[Theorem]{Proposition}

\newtheorem{Remark}[Theorem]{Remark}

\def \dim{{\mbox {dim}}\,}

\def\Z{{\mathbb Z}}
\def\R{{\mathbb R}}

\def \re{{\mathbb R}}

\def \T{{\mathbb T}}

\def \H{{\bar{H}}}

\def \0{\lambda_{0}}

\def \ga{\gamma}

\def\ep{{\epsilon}}

\begin{document}
\title[Stability of critical hypersurfaces]{On the stability of Ma\~n\'e critical hypersurfaces}

\author[L. Macarini]{Leonardo Macarini}
\address{Universidade Federal do Rio de Janeiro, Instituto de Matem\'atica,
Cidade Universit\'aria, CEP 21941-909 - Rio de Janeiro - Brazil}
\email{leonardo@impa.br}

\author[G.P. Paternain]{Gabriel P. Paternain}
 \address{ Department of Pure Mathematics and Mathematical Statistics,
University of Cambridge,
Cambridge CB3 0WB, UK}
 \email {g.p.paternain@dpmms.cam.ac.uk}


\date{October 2009}


\begin{abstract} We construct examples of Tonelli Hamiltonians on $\T^n$
(for any $n\geq 2$) such that the hypersurfaces corresponding to the
Ma\~n\'e critical value are stable (i.e. geodesible). We also provide a criterion for instability
in terms of closed orbits in free homotopy classes and we show that any stable energy level of a Tonelli Hamiltonian must contain a closed orbit.

\end{abstract}

\maketitle

\section{Introduction}

Let $M$ be a closed manifold and $H:T^*M\to \re$ a Tonelli Hamiltonian, that is, a smooth function
that is strictly fibrewise convex and superlinear. The latter means that for all $x\in M$
and some (and hence any) Riemannian metric
\[\lim_{|p|\to\infty}\frac{H(x,p)}{|p|}=\infty.\]
Given a covering $\Pi:\widehat{M}\to M$ we can associate to $H$ a Ma\~n\'e critical value as follows.
We consider the lift $\widehat{H}$ of $H$ to $\widehat{M}$ and we set
\[c(\widehat{H}):=\inf_{u\in C^{\infty}(\widehat{M},\re)}\sup_{x\in\widehat{M}}\widehat{H}(x,d_{x}u).\]
If $\bar{M}$ is a covering of $\widehat{M}$, then clearly $c(\bar{H})\leq c(\widehat{H})$ and equality holds if $\bar{M}$ is a {\it finite} cover of $\widehat{M}$. If $\widetilde{M}$ is the universal covering of $M$, we will denote by $c_u(H)$ the corresponding critical value. If we consider the abelian cover
given by the kernel of the Hurewicz map $\pi_{1}(M)\mapsto H_{1}(M,\re)$ we obtain what is called
the {\it strict} critical value. We will denote it by $c_0(H)$ and it coincides with
$-\beta(0)$, where $\beta:H_{1}(M,\re)\to\re$ is Mather's minimal action function in homology \cite{M}.
Clearly $c_u(H)\leq c_0(H)$ but in general the inequality could be strict (this fact was first pointed out in \cite{PP}, but see \cite{CFP} for many more examples). On the other hand, $c_u(H)=c_0(H)$
as soon as $\pi_1(M)$ is amenable \cite{FM}.

Let $e$ denote the smallest value of $k$ such that $\Sigma_k:=H^{-1}(k)$ intersects every fibre
of $T^*M$. We will suppose throughout that the critical values $c_u(H)$ and $c_0(H)$ are strictly
bigger than $e$; this ensures in particular that they are regular values of
$H$. This happens rather frequently, e.g. if $H$ has the form $H(x,p)=\frac{1}{2}|p+\theta_x|^2$, where
$\theta$ is a 1-form which is not closed.

The critical values have significance from the point of view of the symplectic topology of the
hypersurfaces $\Sigma_k$ \cite{CFP,PPS}. For example, it is well known that if $k>c_{0}(H)$, then
$\Sigma_k$ is of contact type. It is also known that if $M\neq \mathbb T^2$, then $\Sigma_k$
is not of contact type for $k\in [c_u(H),c_0(H)]$ (cf. \cite[Theorem B.1]{Co}) and it is an open
problem to show that in fact $\Sigma_k$ is never of contact type for $k\in (e,c_u(H))$.
Evidence in favour of a positive answer to this problem when $\dim M=2$ is given in \cite{CMP}.
For the case of $\T^2$ the situation is a bit exceptional due to the fact that is the only case where the
projection map $H_{1}(\Sigma_k,\re)\mapsto H_{1}(M,\re)$ is not injective. Exploiting this fact, an example is given in \cite{CMP} with the property that the Ma\~n\'e critical hypersurface
$\Sigma_{c}$, where $c=c_u(H)=c_0(H)$, is of contact type. On the other hand $\Sigma_c$ can never
be of {\it restricted} contact type.

This paper is prompted by the recent interest in a weaker condition than contact type, namely, the {\it stability} or {\it geodesibility} of these
hypersurfaces \cite{CFP}. Recall that $\Sigma_k$ is said to be stable if there exists a smooth
1-form $\lambda$ such that vectors $v\neq 0$ tangent to the characteristic foliation of
$\Sigma_k$ annihilate $d\lambda$ but $\lambda(v)\neq 0$. This is equivalent to saying that the
characteristic foliation of $\Sigma_k$ is geodesible, i.e., there exists a smooth Riemannian metric
on $\Sigma_k$ such that the leaves of the foliation are geodesics of the metric.
The terminology ``stable'' was coined by Hofer and Zehnder \cite{HZ} who introduced
the notion via another equivalent definition: the hypersurface $\Sigma_k$ is stable
if a neighborhood of it can be foliated by hypersurfaces whose characteristic foliations are
conjugate. We remark that in general these nearby hypersurfaces {\it do not} need to coincide with
$H^{-1}(r)$ for $r$ near $k$.

The question that arises now is: if $M\neq \T^2$, can the Ma\~n\'e critical hypersurfaces $\Sigma_{c_u}$ and $\Sigma_{c_0}$ ever be stable? To motivate our result, let us consider an illustrative class
of Hamiltonians first. Suppose $H(x,p)=\frac{1}{2}|p+\theta_x|^2$, where the 1-form
$\theta$ has the property that $|\theta_{x}|=1$ for all $x\in M$. Moreover, suppose
that the vector field $X$ on $M$ metric-dual to $\theta$ has an invariant probability
measure $\mu$ with zero homology, that is,
\[\int_{M}\omega(X)\,d\mu=0\]
for any closed 1-form $\omega$. This is equivalent to saying that the flow of $X$ has no cross-section \cite{S}.
Since the zero section $p=0$ sits inside $\Sigma_{1/2}$, we see that $c_{0}(H)\leq 1/2$.
The dynamics of the characteristic foliation on the zero section coincides with that of $X$.
The condition that $X$ has an invariant probability measure $\mu$ with zero homology
forces $c_0(H)=1/2$. To see this, let $\xi$ be the Liouville 1-form of $T^*M$. Since
$\xi$ vanishes on the zero section and the characteristic foliation has an invariant measure with zero
homology supported on $p=0$, $\Sigma_{1/2}$ cannot be of contact type. But we know that $\Sigma_k$
is of contact type for any $k>c_{0}(H)$, thus $c_{0}(H)=1/2$.
We already pointed out that the dynamics of $X$ sits inside the dynamics of the characteristic 
foliation, thus if $X$ is not geodesible, $\Sigma_{1/2}$ cannot be stable. This provides a large class
of Tonelli Hamiltonians with critical unstable levels. However, even if $X$ is geodesible,
$\Sigma_{1/2}$ is unstable in all the known examples, even in cases in which the dynamics of
$X$ is as simple as a circle action (cf. \cite[Section 6]{CFP}), and thus it is unavoidable to 
speculate that maybe the Ma\~n\'e critical hypersurfaces are always unstable if $M\neq \T^2$.
In the present paper we will exhibit a Tonelli Hamitonian on $\T^n$ for any $n>1$ for which
$\Sigma_c$ is stable ($c=c_0=c_u$). We will also give a criterion for instability of $\Sigma_{c_u}$
based on Contreras' Theorem D in \cite{Co}. More precisely we show:

\medskip

\noindent{\bf Theorem.} 
{\it If there is a free homotopy class $\Gamma$ in $M$ (including the trivial one) such that there
 is no closed orbit with energy $c_u$ whose projection to $M$ belongs to $\Gamma$, then the 
hypersurface $\Sigma_{c_{u}}$ is not stable. Morever, any stable energy level of a Tonelli Hamiltonian must contain a closed orbit.
There are examples of Tonelli Hamiltonians on
$\T^n$ ($n\geq 2$) such that $\Sigma_{c}$ is stable (for the examples $c=c_u=c_0$).
}

\medskip

The examples are peculiar because stability of $\Sigma_{c_{0}}$ implies that there is a stabilizing 1-form $\lambda$ which is a contact form, but since $\Sigma_{c_{0}}$ cannot be of contact type ($n\geq 3$), $d\lambda$ is not a constant multiple of the canonical symplectic form $d\xi$. Indeed, stability gives a thickening of $\Sigma_{c_{0}}$ by hypersurfaces with conjugated characteristic foliations and the hypersurfaces above $\Sigma_{c_{0}}$ (that is, the hypersurfaces that bound a region that contains $\Sigma_{c_{0}}$) are of contact type. The pullback of the contact form by the conjugation gives the form $\lambda$. Hence we have examples of odd-dimensional manifolds $\Sigma$ with two different exact Hamiltonian structures on $\Sigma$ which share the same characteristic foliation.

The examples have a closed contractible orbit in $\Sigma_{c}$ such that any capping disk has zero symplectic area, thus
these hypersurfaces are {\it not} tame as defined in \cite{CFP} (the tameness condition was important to define an invariant Rabinowitz Floer homology for stable levels). By iterating the loop we see
that we also violate the Palais--Smale condition with energy $c$ on the space of contractible loops.
Thus, these examples show that there is no hope of extending to the stable case Contreras' result in \cite{Co} which asserts that a contact-type energy level must satisfy the Palais--Smale condition, unless additional conditions are added.

The proof of the theorem is based on two elementary technical steps that might be useful for further developments. The first one, stated in Lemma \ref{lemma:change}, shows that the thickening induced by the stabilizing 1-form of a stable level $\Sigma$ of a Tonelli Hamiltonian $H$ can be realized by energy levels of another Tonelli Hamiltonian $\tilde H$. Moreover, the energy levels of $\tilde H$ coincide with those of $H$ outside a neighborhood of $\Sigma$.
The second step, in turn, is Proposition \ref{prop:stable} (see also Remark \ref{rmk:prop}) and it establishes that, under suitable conditions, the ``convex suspension'' of a stable level of a Hamiltonian is also stable. More precisely, let $H$ be a Hamiltonian defined on any symplectic manifold $V$ and consider the Hamiltonian $\bar H: V \times T^*S^1 \to \R$ given by the sum of $H$ and the kinetic energy on $T^*S^1$. Suppose that the level $H^{-1}(k)$ is stable and its stabilizing 1-form admits an extension $\alpha$ such that, for every $r$ close enough to $k$, $\alpha|_{H^{-1}(r)}$ is a stabilizing 1-form and $X_H|_{H^{-1}(r)}$ is a constant multiple of the Reeb vector field (for example, the level $\Sigma$ of the Hamiltonian $\tilde H$ given by the first step). Then $\bar H^{-1}(k)$ is stable as well.

\section{Preliminary results}

We start by showing:

\begin{Lemma} Let $H:T^*M\to\re$ be a Tonelli Hamiltonian
and suppose that the regular energy hypersurface $\Sigma_{k}=H^{-1}(k)$
is stable with stabilizing 1-form $\lambda$.
Then, there exists $\ep > 0$ and a Tonelli Hamiltonian $\tilde{H}$ such that:
\begin{enumerate}
\item $\tilde{H}^{-1}(r) = H^{-1}(r)$ for every $r\leq k-\ep$, $\tilde{H}^{-1}(e^{Ar}+B) = H^{-1}(r)$ for every $r\geq k+\ep$ and $\tilde H^{-1}(e^{Ak}+B) = H^{-1}(k) = \Sigma_k$, where $A$ and $B$ are positive constants;
\item there exists $0<\delta<\ep$ such that $\lambda$ extends to $\tilde{H}^{-1}(e^{A(k-\delta)}+B,e^{A(k+\delta)}+B)$
with the property that $\lambda(X_{\tilde{H}})$ is positive
and constant on $\tilde{H}^{-1}(e^{A(k+r)}+B)$ for each $r\in (-\delta,\delta)$.
Moreover, the characteristic foliations of $\tilde{H}^{-1}(e^{A(k+r)}+B)$ are all conjugate;
\item $i_{X_{\tilde{H}}}d\lambda=0$ on
$\tilde{H}^{-1}(e^{A(k-\delta)}+B,e^{A(k+\delta)}+B)$.
\end{enumerate}
\label{lemma:change}
\end{Lemma}

\begin{proof}
Let $\Sigma = \Sigma_k$ and notice that, since $k$ is a regular value, there exists a neighborhood $U$ of $\Sigma$ and a diffeomorphism $\phi: \Sigma \times (-\epsilon,\epsilon) \to U$ such that $\phi(\Sigma \times \{r\}) = H^{-1}(k+r)$ for every $r \in (-\epsilon,\epsilon)$.

Let $\omega$ be the canonical symplectic form of $T^*M$. Consider on $\Sigma \times (-\epsilon,\epsilon)$ the closed 2-form $\widetilde\omega := \omega|_\Sigma + d(r\lambda)$ and notice that it is symplectic whenever $\epsilon$ is sufficiently small. By the coisotropic neighborhood theorem, there exists $0 < \delta < \epsilon$ and a diffeomorphism $\psi: \Sigma \times (-\delta,\delta) \to V \subset U$ such that $\psi^*\omega = \widetilde\omega$ and $\psi|_{\Sigma}$ is the identity. This gives rise to the extension $\psi_*(\lambda)$ of $\lambda$ to $V$ and a thickening of $\Sigma$ with conjugated characteristic foliations given by $\psi(\Sigma \times \{r\})$. For any Hamiltonian $H$ constant on every $\psi(\Sigma \times \{r\})$ we have that $\psi_{*}(\lambda)(X_H)$ is constant on every $\psi(\Sigma \times \{r\})$ as well and $i_{X_{H}}d(\psi_{*}\lambda)=0$
on $V$.

Given a smooth function $f: \Sigma \to (-\epsilon,\epsilon)$ define $\Sigma_f = \{\phi(x,f(x)) \in U, \forall x \in \Sigma\}$. Consider the family of functions $f_r: \Sigma \to \R$, $r \in (-\delta,\delta)$, given by the relation
$$ \Sigma_{f_r} = \psi(\Sigma \times \{r\}). $$
Clearly $f_r$ is $C^\infty$ small and $\partial_r f_r(x) >0$ for every $(x,r) \in \Sigma \times (-\delta,\delta)$. We will construct a smooth family $g_r: \Sigma \to \R$, $r \in (-\epsilon,\epsilon)$, of $C^\infty$ small functions such that $\partial_r g_r(x) > 0 $ for every $(x,r) \in \Sigma \times (-\ep,\ep)$, $g_r(x) = r$ for $|r|$ close to $\epsilon$ and $g_r = f_r$ for $|r|$ small enough.

For this purpose, we will make use of three auxiliary smooth functions. The first one $\alpha: \R \to \R$ is depicted in Figure \ref{fct_alpha}. It is even and satisfies $\alpha(r) = 0$ if $0\leq r \leq \delta_1$, where $\delta_1$ is close to $\delta$ and satisfies $\delta_1<\delta < \ep$, $\alpha^\prime(r) > 0$ for every $\delta_1 < r < \epsilon_1$, where $\delta<\epsilon_1<\ep$ is close to $\ep$, and $\alpha(r) = 1$ if $r\geq\epsilon_1$.  

\begin{figure}[ht]
\begin{center}
\includegraphics[width=3in]{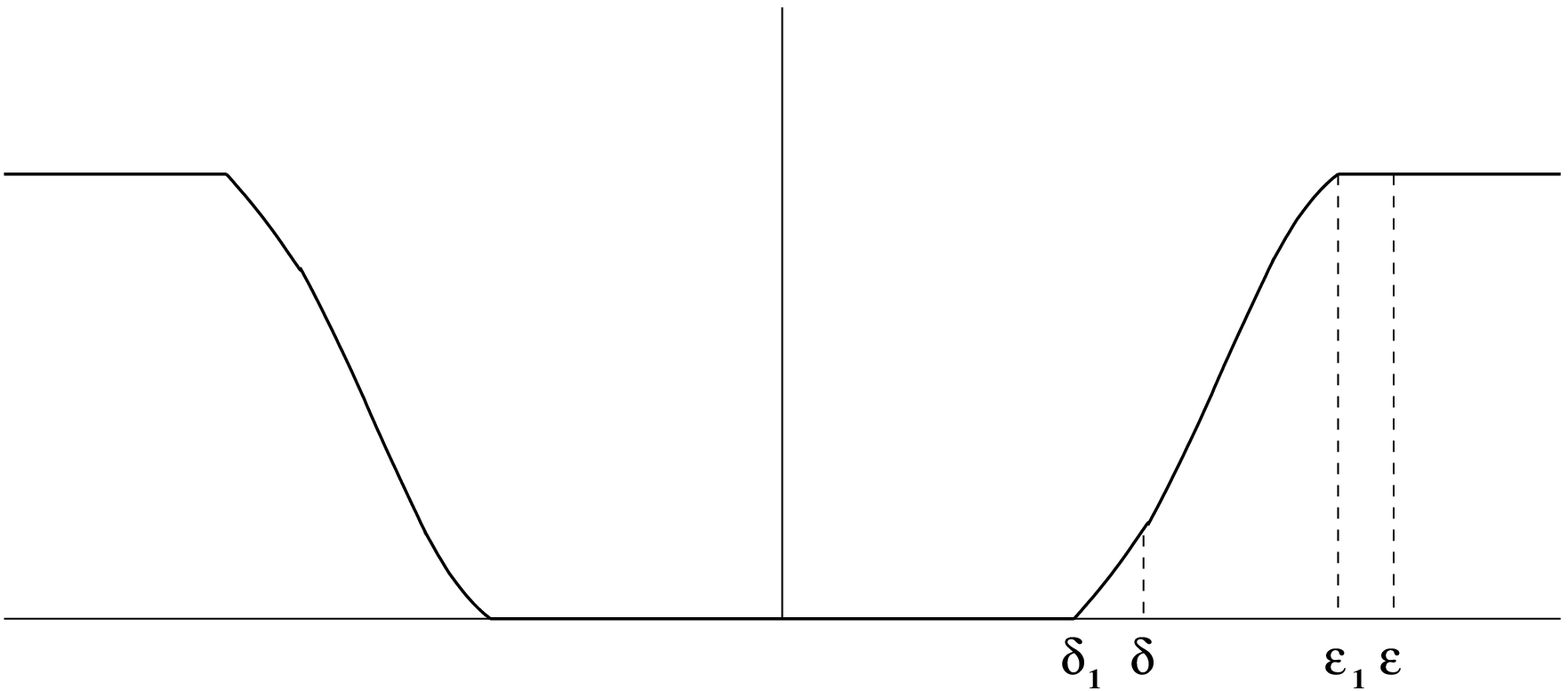}
\caption{\label{fct_alpha} Function $\alpha$.}
\end{center}
\end{figure}

The second and third functions $\beta_0: \R \to \R$ and $\beta_1: \R \to \R$ are outlined in Figures \ref{fct_beta0} and \ref{fct_beta1} respectively. Both are odd and have non-negative derivatives. Choose a constant $\epsilon_2$ close to $\epsilon_1$ satisfying $\delta < \epsilon_2 < \epsilon_1$. The function $\beta_0$ is the identity for $r \geq \epsilon_1$, constant and bigger than $\epsilon_2$ for $\delta_1 \leq r\leq \epsilon_2$ and  $0<\beta_0^\prime(r)<1$ if $\epsilon_2 < r < \epsilon_1$.

\begin{figure}[ht]
\begin{center}
\includegraphics[width=2.5in]{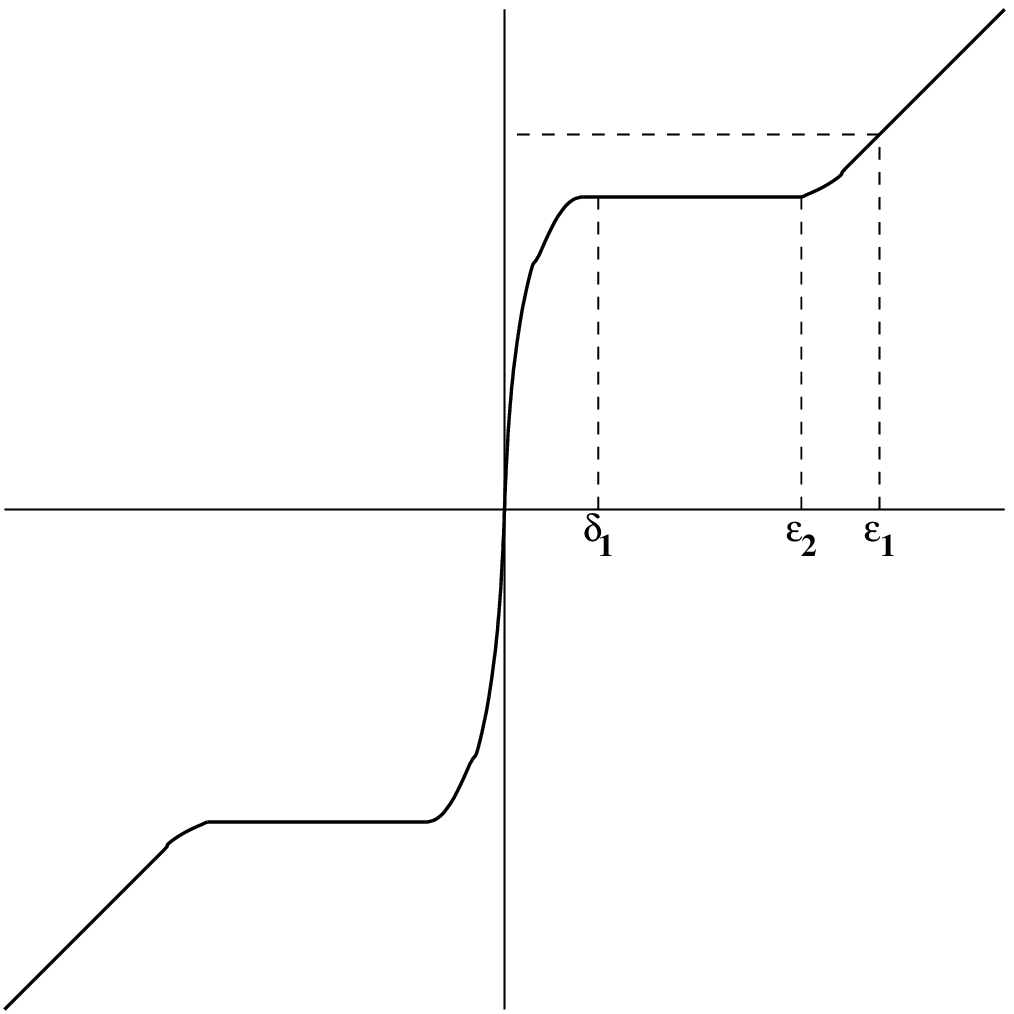}
\caption{\label{fct_beta0} Function $\beta_0$.}
\end{center}
\end{figure}

The function $\beta_1$ is the identity for $0 \leq r \leq \delta_1$, constant and less than $\delta$ for $r\geq \delta$ and its derivative is positive and less than one for $\delta_1 < r < \delta$.

\begin{figure}[ht]
\begin{center}
\includegraphics[width=2.5in]{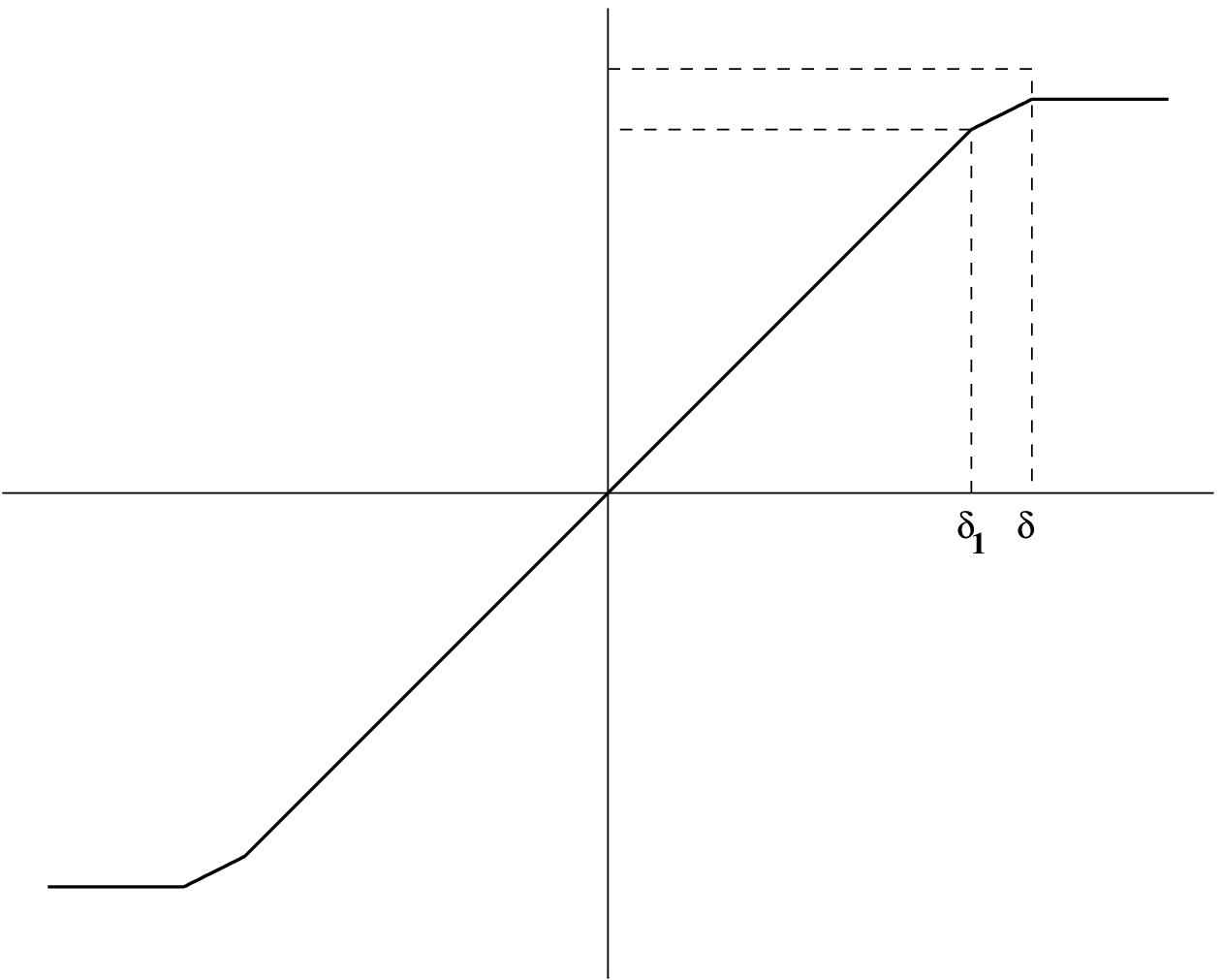}
\caption{\label{fct_beta1} Function $\beta_1$.}
\end{center}
\end{figure}


Now, define
$$ g_r(x) = \alpha(r)\beta_0(r) + (1-\alpha(r))f_{\beta_1(r)}(x). $$
By the properties of the auxiliary functions, $g_r = f_r$ for $-\delta_1 \leq r \leq \delta_1$ and $g_r \equiv r$ for $|r|\geq\epsilon_1$. Moreover, choosing $\delta$ sufficiently small one can make every $g_r$ arbitrarily $C^\infty$ small. It remains to show that $\partial_r g_r>0$ for every $\delta_1 < |r| < \epsilon_1$. We have that
$$ \partial_r g_r(x) = \alpha^\prime(r)(\beta_0(r)-f_{\beta_1(r)}(x)) + (1-\alpha(r))\beta_1^\prime(r)\partial_r f_{\beta_1(r)}(x) + \alpha(r)\beta_0^\prime(r). $$
The second and third terms are clearly non-negative. Since $|f_r(x)| < \ep$ for every $(x,r) \in \Sigma \times (-\delta,\delta)$ and $\epsilon_2$ is arbitrarily close to $\ep$, we can choose it such that $\beta_0(r)-f_{\beta_1(r)}>0$ for $r> \delta_1$ and $\beta_0(r)-f_{\beta_1(r)}<0$ for $r< -\delta_1$. But $\alpha^\prime(r)>0$ if $\delta_1 < r < \epsilon_1$ and $\alpha^\prime(r)<0$ if $-\epsilon_1 < r < -\delta_1$, implying that the first term is positive.

The property $\partial_r g_r>0$ enables us to apply the implicit function theorem to the map
$$ F(x,r,s) = r - g_s(x) $$
to get a smooth function $\bar H: U \to \R$ such that
$$ F(x,r,\bar H(x,r)) = 0 $$
satisfying $\partial_r\bar H(x,r) = -(\partial_s F(x,r,H(x,r)))^{-1}\partial_r F(x,r,H(x,r)) = (\partial_s g_{s}((x))^{-1} > 0$, where $s=\bar H(x,r)$. Note that $\bar H(x,r)$ is the inverse of the function $r \mapsto g_r(x)$ with $x$ fixed. Hence, $\bar H(\Sigma_{g_r}) = r$ for $-\ep < r < \ep$. Extend $\bar H$ to $T^*M$ setting $\bar H(x) = H(x)$ for $x \notin U$. By the construction, $\bar H$ is a smooth function on $T^*M$.

Now, we need to arrange $\bar H$ to make it Tonelli. In order to do it, fix a function $h: \R \to \R$ such that $h^\prime(r)\geq 1$ for every $r$ and define $\tilde H = h \circ \bar H$. Since $\bar H$ coincides with $H$ outside $U$ and $h^\prime(r)\geq 1$, $\tilde H$ is clearly superlinear. We claim that we can choose $h$ such that $\tilde H$ is fiberwise convex. Moreover, it satisfies $h(r) = r$ for $r\leq k-\ep$ and $h(r) = e^{Ar}+B$ for $r\geq k-\ep_1$, where $A$ and $B$ are positive constants.

As a matter of fact, consider a regular energy level $\Sigma$ of $\tilde H$ and a fiber $T^*_x M$ for some $x \in M$. Fix a flat metric on $T^*_x M$ and set $\tilde H_x = \tilde H|_{T^*_x M}$. Since $H$ is fiberwise (strictly) convex and $g_r$ is $C^2$ small for every $r \in (-\ep,\ep)$, $\Sigma_x := \Sigma \cap T^*_x M$ is a hypersurface with positive sectional curvature. But the Hessian of $\tilde H_x$ restricted to $T_p\Sigma_x$ is the second fundamental form of $\Sigma_x$ at $p$ with respect to $-\nabla \tilde H_x(p)$ for every $p \in \Sigma_x$. Thus the Hessian $d^2\tilde H_x$ restricted to $T_p\Sigma_x$ is positive definite.

Consider a function $h$ such that $h^{\prime\prime}(r)\geq 0$ for every $r$, $h(r) = r$ for $r\leq k-\ep$ and $h(r) = e^{Ar}+B$ for $r \geq k-\ep_1$ where $A$ and $B$ are positive constants and $A$ fulfills the condition
$$ A > \max_{(x,p) \in H^{-1}[k-\ep_1,k+\ep_1]} \frac{|d^2\bar H_x(\nabla\bar H_x(p),\nabla\bar H_x(p))|}{d \bar H_x(\nabla\bar H_x(p))}. $$
See Figure \ref{fct_h}. We have that
$$ d^2\tilde H(\nabla\bar H_x(p),\nabla\bar H_x(p)) = h^{\prime\prime}(\bar H(x,p))d\bar H(\nabla \bar H_x(p)) + h^{\prime}(\bar H(x,p))d^2\bar H(\nabla\bar H_x(p),\nabla\bar H_x(p)). $$
\begin{figure}[ht]
\begin{center}
\includegraphics[width=3in]{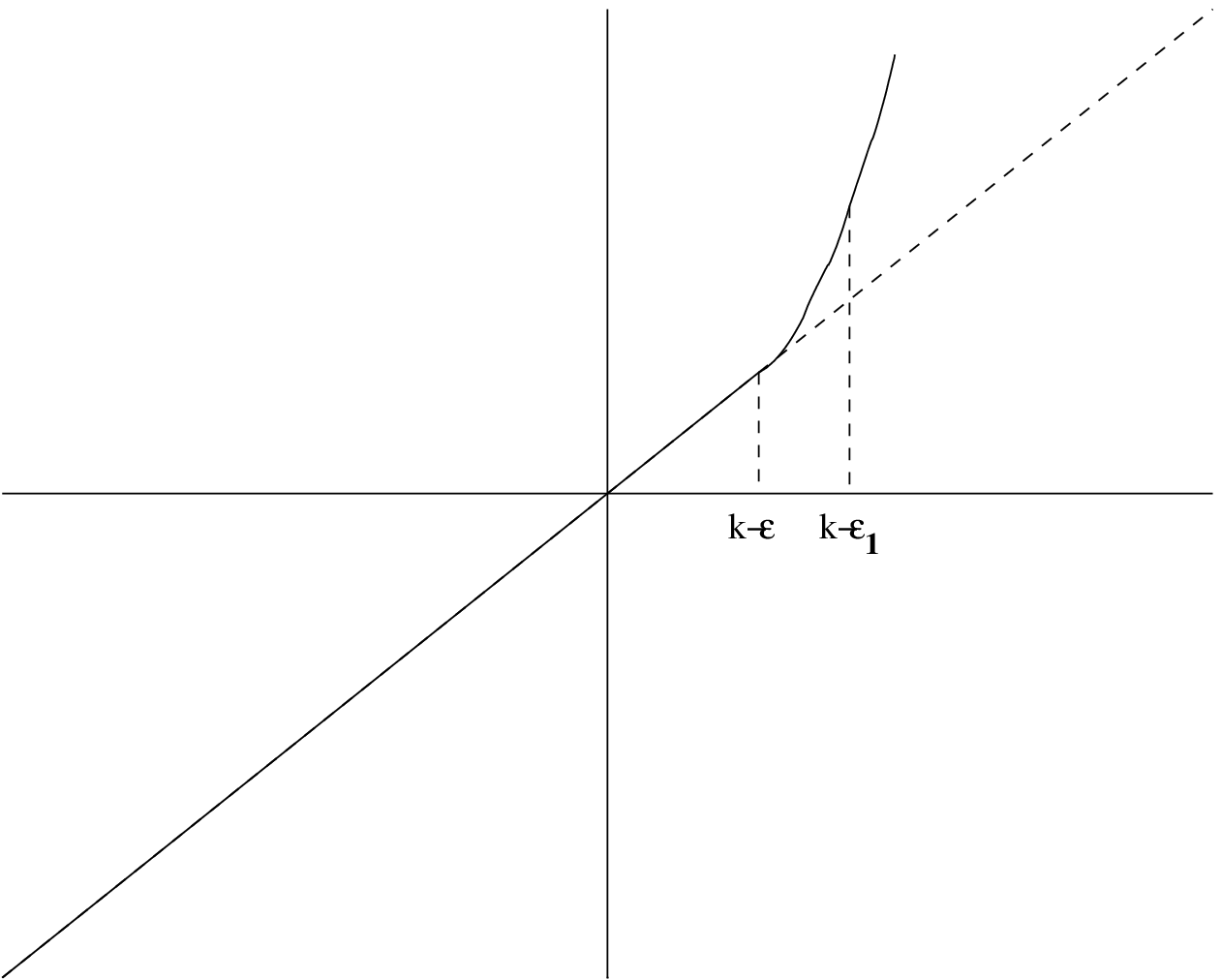}
\caption{\label{fct_h} Function $h$.}
\end{center}
\end{figure}

Since $\bar H$ is convex on $T^*M\setminus H^{-1}(k-\ep_1,k+\ep_1)$ and $h^{\prime\prime}(r) \geq 0$ for every $r$ (which in turn implies that $h^\prime\geq 1$), we conclude that $\tilde H$ is fiberwise convex on $T^*M\setminus H^{-1}(k-\ep_1,k+\ep_1)$. For $(x,p) \in H^{-1}(k-\ep_1,k+\ep_1)$ the condition on $A$ ensures that
$$ h^{\prime\prime}(\bar H(x,p)) = Ah^{\prime}(\bar H(x,p)) > -\frac{d^2\bar H(\nabla\bar H_x(p),\nabla\bar H_x(p))}{d\bar H(\nabla \bar H_x(p))}h^{\prime}(\bar H(x,p)), $$
finishing the proof of the lemma.
\end{proof}

Let $M$ be a closed manifold and $H:T^*M\to \re$ a Tonelli Hamiltonian.
On $N:=M\times S^1$, consider the Tonelli Hamiltonian
$\H(x,p,t,p_t):=H(x,p)+\frac{1}{2}p_{t}^2$, where $(x,p)\in T^*M$ and
$(t,p_t)\in T^*S^1$.

Recall that for a Tonelli Hamiltonian the strict Ma\~n\'e critical value
is given by
\[c_0(H)=\inf_{\theta}\max_{x\in M}H(x,\theta_x),\]
where $\theta$ runs over all smooth closed 1-forms in $M$.

\begin{Lemma} $c_{0}(\H)=c_0(H)$.

\label{lemma:eqc}
\end{Lemma}

\begin{proof} Let $\tau:M\times S^1\to M$ be the first factor projection.
If $\theta$ is a closed 1-form in $M$, then
\[\H(x,t,(\tau^*\theta)_{(x,t)})=H(x,\theta_{x})\]
and the inequality $c_0(\H)\leq c_{0}(H)$ follows immediately.

Now let $\theta$ be an arbitrary closed 1-form in $M\times S^1$.
Write $\theta_{(x,t)}=a_{(x,t)}+b_{(x,t)}$, where
$a_{(x,t)}\in T_{x}^*M$ and $b_{(x,t)}\in T_{t}^*S^1$. Then
\[\H(x,t,\theta_{(x,t)})=H(x,a_{(x,t)})+\frac{1}{2}b_{(x,t)}^2\geq
H(x,a_{(x,t)}).\]
Fix $t_0\in S^1$. Since $x\mapsto a_{(x,t_0)}$ is a closed 1-form
in $M$ we have
\[\max_{(x,t)\in M\times S^1}\H(x,t,\theta_{(x,t)})\geq \max_{x\in M}
\H(x,t_0,\theta_{(x,t_0)})\geq \max_{x\in M} H(x,a_{(x,t_0)})\geq c_0(H).\]
Hence
\[c_{0}(\H)\geq c_0(H)\]
and the lemma is proved.

\end{proof}

Let $k$ be a regular value of $H$ and set $\Sigma_k:=H^{-1}(k)$.

\begin{Proposition} Suppose there is a smooth 1-form $\alpha$ defined
in a neighborhood $H^{-1}(k-\delta,k+\delta)$ 
of $\Sigma_k$ such that:
\begin{itemize}
\item $i_{X_{H}}d\alpha=0$;
\item $\alpha(X_{H})$ is constant and positive on $\Sigma_{r}$
for all $r\in (k-\delta,k+\delta)$.
\end{itemize}
Then, the hypersurface $\bar{\Sigma}_{k}=\H^{-1}(k)$ is stable.
\label{prop:stable}
\end{Proposition}

\begin{proof} Take $\varepsilon$ such that $\varepsilon<\sqrt{2\delta}$
and consider a smooth function $f:\re\to\re$ with the following properties:
\begin{itemize}
\item $f$ is even and has support in $(-\varepsilon, \varepsilon)$;
\item $f\geq 0$;
\item $f(s)\equiv 1$ for $s\in (-\varepsilon/2, \varepsilon/2)$;
\item $f'(s)\leq 0$ for $s\geq 0$.
\end{itemize}
Sine we are assuming that $\alpha(X_{H})$ is a function
of one variable only, we let
\[r(s):=\alpha(X_{H})(k-s^{2}/2).\]
Notice that $r(s)$ is the value of $\alpha(X_H)$ on the level of energy $k-s^{2}/2$. This function is defined on $(-\varepsilon, \varepsilon)$.
Consider now the smooth function $g:\re\to\re$ defined by
\begin{equation}
g(s):=-\int_{0}^{s}\frac{r(u)f'(u)}{u}\,du.
\label{eq:defg}
\end{equation}
Note that the integrand is defined and smooth in all $\re$
since $f$ has support in  $(-\varepsilon, \varepsilon)$ and
$f'=0$ in  $(-\varepsilon/2, \varepsilon/2)$.

Let $\pi:T^*(M\times S^1)\to T^*M$ be the obvious projection.
The 1-form $\pi^*\alpha$ is defined only on $(H\circ\pi)^{-1}(k-\delta,k+\delta)$, but the 1-form $f(p_{t})\pi^*\alpha$ is smooth and defined in
all $\bar{\Sigma}_k$.

We claim that the 1-form
\[\lambda:=f(p_t)\,\pi^*\alpha+g(p_t)\,dt\]
stabilizes $\bar{\Sigma}_{k}$. In order to check the claim note first that
\[X_{\H}=X_{H}+(0,p_t,0,0).\]
Thus
\[\lambda(X_{\H})=f(p_t)\,\alpha(X_{H})+g(p_{t})\,p_t
=f(p_t)r(p_t)+g(p_{t})\,p_t.\]
By our choice of $f$, $fr\geq 0$ and $f(0)r(0)>0$.
Also, from \eqref{eq:defg} we see that
$g'\geq 0$ and $g(0)=0$, hence $g(p_t)p_t\geq 0$ for
all $p_t$. Note also that by construction $f$ and $g$ cannot both vanish at the same time. Indeed, $f(s) \equiv 1$ for $s \in (-\ep/2,\ep/2)$, $g\geq 0$, $g^\prime \geq 0$ and $g^\prime(s)>0$ whenever $f^\prime(s)\neq 0$. It follows
that $\lambda(X_{\H})>0$ at every point of $\bar{\Sigma}_{k}$.

It remains to show that $i_{X_{\H}}d\lambda=0$ on $\bar{\Sigma}_{k}$. 
We compute
\[d\lambda=df\wedge\pi^*\alpha+f\,\pi^* d\alpha+dg\wedge dt.\]
Since $df=f'(p_t)\,dp_t$, $dg=g'(p_t)\,dp_t$ we see that
$df(X_{\H})=dg(X_{\H})=0$. Also $dt(X_{\H})=p_t$. Hence
\[i_{X_{\H}}d\lambda=-\alpha(X_{H})f'(p_t)\,dp_t
+f(p_t)\,i_{X_{\H}}\pi^*d\alpha-p_t g'(p_t)\,dp_t.\]
But since $i_{X_{\H}}\pi^*d\alpha=0$ we have
\[i_{X_{\H}}d\lambda=-r(p_t)f'(p_t)\,dp_t
-p_t g'(p_t)\,dp_t.\]
It follows that $i_{X_{\H}}d\lambda=0$ if and only if
\[-g'(p_t)p_t=r(p_t)f'(p_t)\]
which is a consequence of \eqref{eq:defg}.

\end{proof}

\begin{Remark}
\label{rmk:prop}
{\rm Notice that the argument above can be applied to a Hamiltonian defined on any symplectic manifold with a stable energy level satisfying the hypotheses of the proposition.}
\end{Remark}

\section{Proof of the theorem}

The following lemma provides the basis for the iterative construction
of the examples in $\T^n$ for any $n\geq 2$.

\begin{Lemma} Let $H$ be a Tonelli Hamitonian on $\T^2$ and assume
that $c_0(H)$ is a regular value of $H$. Suppose there is a unique
minimizing measure with zero homology which is ergodic.
 Then, the energy
level $c_0(H)$ is of contact type, but not of restricted contact type.
\label{lemma:ccont}
\end{Lemma}

\begin{proof} By considering Hamiltonians of the form $H(x,p+\omega_x)$, where $\omega$ is a closed 1-form, 
we may suppose that $H$ is such that $c=c(H)=c_{0}(H)$.

Consider the Lagrangian $L:T\T^2\to\re$ associated with $H$ via the
Legendre transform $\ell:T\T^2\to T^*\T^2$. The energy $E$ is given by $H\circ\ell$.
We will use the following elementary
relation which holds for any $(x,v)\in E^{-1}(c)$:
\begin{equation}
L(x,v)+c=\xi_{\ell(x,v)}(d\ell(V(x,v)),
\label{eq:rel}
\end{equation}
where $\xi$ is the Liouville 1-form and $V$ is the Euler-Lagrange vector field of $L$.

Let $\mu$ be an invariant Borel probability measure in $\Sigma_c=H^{-1}(c)$ which is exact as a current. This simply means that it has zero homology in $\Sigma_c$: for any closed 1-form $\psi$ in $\Sigma_c$ we have
\[\int_{\Sigma_{c}}\psi(X_{H})\,d\mu=0,\]
where $X_H$ is the Hamiltonian vector field of $H$.
To show that $\Sigma_c$ is of contact type we need to show that (cf. \cite{McD})
\[\int_{\Sigma_{c}} \xi(X_{H})\,d\mu\neq 0.\]
The measure $\nu:=\ell^*\mu$ in $T\T^2$ has zero homology in $\T^2$. By definition of minimizing measure
\[\int (L+c)\,d\nu\geq 0\]
with equality if and only if $\nu$ is minimizing. Using (\ref{eq:rel}) we see that the only
way in which $\Sigma_c$ can fail to be of contact type is if $\nu$ is minimizing.
By \cite[Proposition 2.1]{CMP} the support of $\nu$ is a union of closed orbits, but since
we are assuming $\nu$ ergodic, the support of $\nu$ consists of a unique closed orbit.
By Mather's graph theorem \cite{M}, this orbit projects to $M$ as a {\it simple} closed
curve $\gamma:[0,T]\to\T^2$. The point now is that the curve 
$[0,T]\ni t\mapsto (\gamma(t),\dot{\gamma}(t))\in E^{-1}(c)$ is not homologous to zero in
$E^{-1}(c)$. Hence the measure $\mu$ that we started with, is not exact as a current
in $\Sigma_c$, and $\Sigma_c$ must be of contact type.

Let us show that $\Sigma_c$ is not of restricted contact type. If it is, there is a 1-form
$\tau$ defined in {\it all} $T^*\T^2$ such that $d\tau=d\xi$ and $\tau(X_{H})>0$ on $\Sigma_{c}$. Since 
$d\tau=d\xi$, we can find a closed 1-form $\varphi$ on $\T^2$ and a smooth function
$f:T^*\T^2\to\re$ such that $\xi=\tau+\pi^*\varphi+df$, where $\pi:T^*\T^2\to\T^2$
is the canonical projection.
Consider a minimizing measure $\nu$ with zero homology in $\T^2$. If we let
$\mu:=\ell_{*}\nu$, we have
\[\int_{\Sigma_{c}}(\pi^*\varphi+df)(X_{H})\,d\mu=0\]
and thus
\[\int_{\Sigma_{c}}\xi(X_{H})\,d\mu=\int_{\Sigma_{c}}\tau(X_{H})\,d\mu>0\]
which by (\ref{eq:rel}) contradicts
\[\int(L+c)\,d\nu=0.\]

\end{proof}

\begin{Remark}{\rm For any $k>c_0(H)$, $\Sigma_k$ is of restricted contact type.
The results in \cite{CMP} together with the argument at the end of the proof of the lemma,
show that for any Tonelli Hamiltonian $H$ on $\T^2$ of the form $H(x,p)=\frac{1}{2}|p+\theta_x|^2$, where
$\theta$ is a 1-form which is not closed, 
and $0<k\leq c_{0}(H)$, the hypersurface $\Sigma_k$ is not of restricted contact type.
}
\end{Remark}

\begin{Remark}{\rm It is easy to give examples of Tonelli Hamiltonians in $\T^2$ with only
one minimizing measure with zero homology which is ergodic. An explicit example
appears in \cite[Example 5.1]{CMP}. We recall it here for the reader's convenience.
Let $\langle\cdot,\cdot\rangle$ be the flat metric.
Consider a smooth vector field $Z$ on $\T^{2}$ such that $Z$ has a simple closed
orbit $\ga$ homotopic to zero and with speed one with respect to the
flat metric.
Take a $C^{\infty}$ function $\psi:\T^{2}\to \re$ such that
$\psi(x)\geq 0$ and $\psi(x)=0$ iff $x\in\ga$.
Set $\theta_{x}(v):=\langle Z(x),v\rangle$ and
$\varphi(x):=|Z(x)|^{2}+2\psi(x)$.
Our Lagrangian will be:
\[L(x,v)=\frac{1}{2}\varphi(x)|v|^{2}-\theta_{x}(v).\]
An easy computation shows that
\[L(x,v)+\frac{1}{2}=\frac{1}{2}\varphi(x)\left|v-\frac{Z(x)}{\varphi(x)}\right|^{2}+\frac{\psi(x)}{\varphi(x)}.\]
It follows that $L(x,v)+1/2\geq 0$ with equality iff $x\in\ga$
and $v=Z(x)$ and therefore $\ga$ is the projection of a closed orbit.
The probability measure associated with $\ga$ is the only minimizing measure
with zero homology. In particular, it follows that
$c_{0}(H)=1/2$, where $H$ is the Hamiltonian convex dual to $L$.

}
\label{remark:example}
\end{Remark}

\noindent{\it Proof of the theorem.} Suppose $\Sigma_c$ is stable. We will show that
there are closed orbits with energy $c$ in every homotopy class $\Gamma$.
By Lemma \ref{lemma:change} with $k=c$ we can consider the new Tonelli Hamiltonian
$\tilde{H}$. Since $H^{-1}(c)=\tilde{H}^{-1}(e^{Ac}+B)=\Sigma_c$, we deduce immediately
from the definition of the critical value that $c_u(\tilde{H})=e^{Ac}+B:=d$.
The key property of $\tilde{H}$ given by item 2 in Lemma \ref{lemma:change} is that
for any $r$ sufficiently close to $d$, the characteristic foliation of $\tilde{\Sigma}_r$ is conjugate
to the one of $\Sigma_c$.

Suppose first $\Gamma\neq 0$. By \cite[Theorem 27]{CIPP2} (see also \cite{Co}), for any $r>d$,
the energy level $\tilde{H}^{-1}(r)$ possesses a closed orbit whose projection to $M$ belongs
to $\Gamma$. Since the characteristic foliations are conjugate for $r$ close to $d$,
the same holds true for $\Sigma_c$.
Suppose now that $\Gamma=0$, then by \cite[Theorem D]{Co} for almost every
$r\in (e^{A(c-\delta)}+B,d)$, the hypersurface $\tilde{H}^{-1}(r)$ carries a closed orbit whose projection
to $M$ is contractible. Again, since the characteristic foliations are conjugate,
the same holds true for $\Sigma_c$.

The argument above clearly shows that if
$\Sigma_k$ is any stable energy level of a Tonelli Hamiltonian, it must contain
a closed orbit.

Let us show the existence of Tonelli Hamiltonians on $\T^n$ for which $\Sigma_c$ is stable.
(Recall that for any Hamitonian on $\T^n$, $c:=c_u(H)=c_0(H)$.)
Using Remark \ref{remark:example} and Lemma \ref{lemma:ccont} there is a Tonelli
Hamiltonian $H$ on $\T^2$, such that $\Sigma_c$ is of contact type and $c$ is a regular
value of $H$. By Lemma \ref{lemma:change} we can replace $H$ by another Hamiltonian with
the same Ma\~n\'e critical hypersurface and critical value $d:=e^{Ac}+B$ (which we still denote by $H$) but such that it satisfies
the hypotheses of Proposition \ref{prop:stable}. Hence there is another
Hamiltonian $\bar{H}$ on $\T^3$ such that $\bar{\Sigma}_d$ is stable.
By Lemma \ref{lemma:eqc}, $H$ and $\bar{H}$ have the same critical value and $\bar{H}$ is 
our desired example in $\T^3$. Iterating this construction we get examples in $\T^n$ for
any $n>1$.

\qed

\subsection{An application of the criterion} We provide an example where the
criterion of the theorem can be easily applied. The example here is different
from the examples in \cite{CFP}.

Let $G={\bf Sol}$ be the semidirect product of $\R^2$ with $\R$, with
coordinates $(x,y,z)$ and multiplication
\begin{equation*} \label{eq:sm}
(x,y,z)\star(x',y',z')=(x+e^{z}x', y+e^{-z}y',z+z').
\end{equation*}
The map $(x,y,z) \mapsto z$ is the epimorphism ${\bf Sol} \to \R$
whose kernel is the normal subgroup $\R^2$.
The group ${\bf Sol}$ is isomorphic to the matrix group
\[\left(\begin{array}{ccc}

e^z&0&x\\
0&e^{-z}&y\\
0&0&1\\

\end{array}\right).\]

It is not difficult to see that ${\bf Sol}$ admits cocompact lattices.
Let $A\in SL(2,\mathbb Z)$ be such that there is $P\in GL(2,\mathbb R)$ with
\[PAP^{-1}=\left(\begin{array}{cc}
\lambda &0\\
0&1/\lambda\\
\end{array}\right)\]
and $\lambda>1$.
There is an injective homomorphism
$$\mathbb Z^2\ltimes_{A}\mathbb{Z}\hookrightarrow \mathbf{Sol}$$
given by $(m,n,l)\mapsto (P(m,n),(\log\lambda)\,l)$ which defines
a cocompact lattice $\Gamma$ in $\mathbf{Sol}$.
The closed 3-manifold $M:=\Gamma\setminus {\bf Sol}$ is a 2-torus
bundle over the circle with hyperbolic gluing map $A$.
The closed 1-form $dz$ generates $H^{1}(M,\re)=\re$ and
the abelian cover $M_0$ of $M$ is given by $\Gamma_0\setminus {\bf Sol}$
where $\Gamma_0\subset \Gamma$ is the copy of $\Z^2$ obtained by setting $l=0$.
The manifold $M_0$ is diffeomorphic to $\T^2\times \re$.

If we denote by $p_{z}$, $p_{x}$ and $p_{y}$ the momenta that
are canonically conjugate to $z$, $x$ and $y$ respectively,
then the functions
\begin{equation*} \label{eq:mom}
\begin{array}{lcl}
M_x &=& e^{z} p_{x}, \\
M_y &=& e^{-z} p_{y}, \\
M_z &=& p_{z}
\end{array}
\end{equation*}
are left-invariant functions on $T^*{\bf Sol}$. The 1-form
$\theta:=e^{-z}dx$ is also left-invariant and we consider the
following left-invariant Tonelli Hamiltonian on ${\bf Sol}$ which
descends to $M$:

\[2H=e^{2z}(p_{x}+e^{-z})^2+e^{-2z}p_{y}^2+p_{z}^2=(M_x+1)^2+M_z^2+M_z^2.\]
Observe that this Hamiltonan is of the form $H(x,p)=\frac{1}{2}|p+\theta_x|^2$,
where the metric is:
\[ds^2=e^{-2z}dx^2+e^{2z}dy^2+dz^2.\]
Obviously $|\theta|=1$. Since the vector field $X$ metric-dual
to $\theta$ has a flow that translates in the $x$-direction we see that
$dz(X)=0$ and therefore any $X$-invariant measure has zero homology. By the argument explained
in the introduction we conclude that $c_0(H)=1/2$. Moreover, since solvable
groups are amenable, $c_u(H)=c_0(H)=1/2$.

\medskip

\noindent{\bf Claim.} All closed orbits in $\Sigma_{1/2}$
are homologous to zero.

\medskip

The Claim shows that a free homotopy class with $l\neq 0$ does not contain
the projection of a closed orbit. Hence the Theorem implies that $\Sigma_{1/2}$
is {\it not} stable.

Let us prove the Claim. Hamilton's equations are:

\begin{equation} \label{eq:XH}
X_H = \left\{ \begin{array}{lclclcl}
\dot{x}   &=& (M_{x}+1)e^z,          & \hspace{10mm} & \dot{M_{x}}   &=& M_{x}M_{z},\\
\dot{y} &=& M_{y}\,e^{-z}, &            & \dot{M_{y}} &=& -M_{y}M_{z},\\
\dot{z} &=& M_{z}, &           & \dot{M_{z}} &=& M_{y}^2 - M_{x}(M_{x}+1).
\end{array}
\right.
\end{equation}
Suppose there is a closed orbit in $\Sigma_{1/2}$ with period $T$ and projection (to $M$) $\gamma$. To show
that it is homologous to zero it suffices to show that
\[\int_{0}^{T}dz(\dot{\gamma})\,dt=\int_{0}^{T}M_{z}\,dt=0.\]
But $m:=M_{x}M_{y}$ is a first integral of $X_{H}$, so if our closed
orbit has $m\neq 0$ then $M_{x}\neq 0$ and from (\ref{eq:XH}) we see that
\[\int_{0}^{T}M_{z}\,dt=\int_{0}^{T}\frac{\dot{M}_{x}}{M_{x}}\,dt=0.\]
If $m=0$ using that the closed orbit lives on $(M_{x}+1)^2+M_{y}^2+M_{z}^2=1$
we deduce that $M_{x}=M_{y}=M_z=0$ and again the orbit is homologous to zero.

\end{document}